\documentclass[12pt]{amsart}
\usepackage{enumerate}
\usepackage{helvet,courier,verbatim}
\usepackage{amsmath,xcolor}
\usepackage{amssymb}

\theoremstyle{plain}
\newtheorem{thm}{Theorem}[section]
\newtheorem{lm}[thm]{Lemma}
\newtheorem{cor}[thm]{Corollary}
\newtheorem{prop}[thm]{Proposition}

\theoremstyle{remark}
\newtheorem{rmk}{Remark}

\theoremstyle{definition}

\newcommand{\bnu}{\begin{enumerate}}
\newcommand{\enu}{\end{enumerate}}

\newcommand{\q}{\quad}
\newcommand{\qq}{\qquad}

\newcommand{\al}{\alpha}

\newcommand{\om}{\omega}
\newcommand{\Om}{\Omega}

\newcommand{\la}{\lambda}

\newcommand{\si}{\sigma}

\newcommand{\vp}{\varphi}
\newcommand{\de}{\delta}

\newcommand{\bbz}{\mathbb{Z}}

\newcommand{\bbr}{\mathbb{R}}

\newcommand{\rn}{\mathbb{R}^n}

\newcommand{\f}{\frac}

\newcommand{\p}{\partial}
\newcommand{\nf}{\infty}

\newcommand{\tf}{\tfrac}
\newcommand{\wh}{\widehat}

\newcommand{\mc}{\mathcal}

\newcounter{question}

\newcommand{\bpf}{\begin{proof}}

\newcommand{\epf}{\end{proof}}

\newcommand{\R}{\mathbb R}
\newcommand{\Z}{\mathbb Z}

\allowdisplaybreaks

\makeindex  

\begin{document}

\begin{abstract}

We obtain a sharp   $L^2\times L^2 \to L^1$ boundedness criterion for a class of 
bilinear  operators associated with a multiplier given by a 
signed sum of dyadic dilations of a given function,  in terms 
of the $L^q$ integrability of this function; precisely we show that 
boundedness holds if and only if $q<4$. 
We discuss   applications of this result concerning  bilinear rough singular integrals
and    bilinear  dyadic  spherical maximal functions. 

Our second result is  
an optimal $L^2\times L^2\to L^1$ boundedness 
criterion for bilinear  operators associated 
with multipliers with $L^\infty$ derivatives.
This result provides the main tool in the proof of the 
first theorem and is also manifested  
 in terms of the $L^q$ integrability of
the multiplier.  The optimal range is $q<4$ which, 
in the absence of Plancherel's identity on $L^1$, should be 
compared to $q=\infty$ in the classical $L^2\to L^2$ boundedness for linear multiplier operators. 
\end{abstract}

\author{Loukas Grafakos, Danqing He,  Lenka Slav\'\i kov\'a}
\address{Department of Mathematics, University of Missouri, Columbia MO 65211, USA}
\email{grafakosl@missouri.edu}
\address{Department of Mathematics
Sun Yat-sen (Zhongshan) University, Guangzhou, Guangdong, China}
\email{hedanqing@mail.sysu.edu.cn}
\address{Department of Mathematics, University of Missouri, Columbia MO 65211, USA}
\thanks{The first author was supported by the Simons Foundation and by 
the University of Missouri Research Board and Council.
The second author was supported by NNSF of China (No. 11701583), Guangdong Natural Science Foundation
(No. 2017A030310054) and the
Fundamental Research Funds for the Central Universities (No. 17lgpy11).}
\email{slavikoval@missouri.edu}
\thanks{2010 Mathematics Classification Number 42B15, 42B20, 42B99}

\title[$L^2\times L^2 \to L^1$
boundedness criteria]{$L^2\times L^2 \to L^1$
boundedness criteria} 
\date{}
\maketitle


\section{Introduction}\label{S:introduction}

A linear multiplier operator has the form 
$$
S_m(f)(x)=  \int_{\R^n} m(\xi ) \wh{f}(\xi)  e^{2\pi ix\cdot \xi }\,d\xi 
$$
where $m$ is a bounded function on $\R^n$ and $f$ is a Schwartz function whose
 Fourier transform is 
defined by $\wh{f}(\xi)= \int_{\R^n} f(x) e^{-2\pi i x \cdot \xi}dx$.  Here $x\cdot y$ is the usual dot product on $\mathbb R^n$. An important question in 
harmonic analysis is to find  sufficient conditions on $m$ for $S_m$ to admit a
bounded extension from $L^p(\R^n)$ to itself for $1< p<\infty$. If this is the case, the function  
$m$ is called an $L^p$ Fourier multiplier. In view of Plancherel's theorem,  $m$ is 
an $L^2$ Fourier multiplier  if and only if it  is an $L^\nf$ function.  

In this work we investigate the  $L^2\times L^2 \to L^1$ boundedness of  
bilinear multiplier operators which  is as  central  in this theory
 as the $L^2$ boundedness is in  linear multiplier theory.  
In the linear case, $S_m$ is bounded on $L^2$ exactly when  $m$ lies in $L^\nf$. However, 
there does not exist such a straightforward characterization  in this situation, due to 
the lack of Plancherel's theorem on $L^1$. This provides a strong motivation  
to search for sharp sufficient conditions for  the $L^2\times L^2 \to L^1$  boundedness 
 of bilinear multiplier operators, i.e., operators that have       the form   
$$
T_m(f,g)(x)=\int_{\R^n}\int_{\R^n} m(\xi,\eta) \wh{f}(\xi) \wh{g}(\eta) e^{2\pi ix\cdot(\xi+\eta)}\,d\xi d\eta\, ,
$$
where $f,g$ is a pair of Schwartz functions and
$m$ is a   bounded function on $\R^{2n}$.

A classical sufficient condition for boundedness of $T_m$   
is the so-called Coifman-Meyer condition \cite{CM} on $m$, namely the requirement that 
\begin{equation}\label{Mikhlin}
|\p^\al m(\xi,\eta)|\le C_\al |(\xi,\eta)|^{-|\al|}
\end{equation}
 for sufficiently   many  $\al$. This condition   implies that  
  $T_m$ is bounded from $L^{p_1}\times L^{p_2}$ to $L^p$  
  for   $1<p_1,p_2<\nf$ when $1/p=1/p_1+1/p_2$; see 
  \cite{CM} for $p\ge 1$ and \cite{GraTo}, \cite{KS} for $1/2<p<1$. In other words, 
  this theorem says that 
 linear Mikhlin multipliers on $\mathbb R^{2n}$ are
  bounded bilinear multipliers on $\mathbb R^{n}\times \mathbb R^{n}$.  
Analogous results for bilinear multipliers  that 
satisfy    H\"ormander's~\cite{Horman} classical weakening of \eqref{Mikhlin} 
for linear operators, expresssed 
in terms of  Sobolev spaces,      
was initiated  by Tomita \cite{TomitaJFA} and was subsequently further investigated by  
Grafakos,   Fujita, Miyachi, Nguyen, Si, and Tomita among others; see 
 \cite{GraSi} \cite{FuTomi},  \cite{GrMiTo}, \cite{MiTo},  \cite{MiTo2},  \cite{GraNg}, 
\cite{GrMiNgTo}. Related to this we highlight that 
if the     functions $m(2^k\cdot)\phi $ have $s$ derivatives 
  in $L^r(\mathbb R^{2n})$ ($1<r<\infty$) 
  uniformly in  $k\in \mathbb Z$, with $\phi$ being a suitable smooth bump 
  supported in $1/2<|\xi|<2$, 
then   $T_m$ is bounded
from $L^2\times L^2$ to $L^1$ when $s>s_0=\max(n/2,2n/r)$ and $s_0$ 
cannot be replaced by any smaller number; see \cite{GraHeHon}.
Thus more than $n/2 $ derivatives  of $m(2^k\cdot)\phi$  in $L^4(\mathbb R^{2n})$ uniformly in $k$ 
are required of a generic 
multiplier $m$  for $T_m$ to map $L^2(\mathbb R^{n})\times L^2(\mathbb R^{n})$ 
to $L^1(\mathbb R^{n})$.

There are, however, many other multipliers  emerging   naturally in the study of bilinear operators 
which do not fall under in the scope of the Coifman-Meyer condition.
For one class of such multipliers, described below, we obtain a full
characterization of  $L^2\times L^2\to L^1$ boundedness.

Let us consider a function $m$ on $\mathbb R^{2n}$ 
which satisfies,   for some     $\delta>0$, 
\begin{equation}\label{e09182}
|m(\xi,\eta)|\le C'\min(|(\xi,\eta)|,|(\xi,\eta)|^{-\de}),
\end{equation}
and 
\begin{equation}\label{e09181}
|\p^\al m(\xi,\eta)|\le C_\alpha \min(1,|(\xi,\eta)|^{-\de})\q 
\end{equation}
for all multiindices $\alpha$. 
Unlike   the case of Coifman-Meyer conditions~\eqref{Mikhlin}, the rate of decay in~\eqref{e09181} does not depend on the order of derivatives,
and $\de$ could be arbitrarily small, which means that conditions \eqref{e09182} and 
\eqref{e09181} are satisfied by a variety of functions
which are not multipliers of Coifman-Meyer type.

Given a function $m$ with   properties~\eqref{e09182} and~\eqref{e09181}, we set 
$$
m_k(\xi,\eta)=m(2^k(\xi,\eta))
$$
 for   $k\in \mathbb Z$ and we define a multiplier 
 \begin{equation}\label{e02161}
 \sum_k r_k m_k,
  \end{equation}
 where $r_k$ is a given bounded sequence. 
 We investigate boundedness properties of the operator
\begin{equation}\label{E:T}
T:=T_{\sum_{k\in \Z} r_k m_k}=\sum_{k\in \Z} r_k T_{m_k}.
\end{equation}
 


 A sufficient and essentially  necessary 
 condition in terms of the 
integrability of $m$   for the boundedness of the operator $T$ from $L^2 \times L^2$ to $L^1$ is given in the following theorem. In the sequel  
$\|T \|_{L^p\times L^q\to L^r}$  denotes the norm of $T$
from $L^p\times L^q\to L^r$, $\left\lfloor b \right\rfloor$
 denotes the integer part of a real number $b$, and
  $\mathcal C^M(\mathbb R^{2n})$ denotes the class of  all
 functions on $ \mathbb R^{2n} $ whose partial derivatives of order up to 
 and including order $M$ are continuous.

\begin{thm}\label{04161}
Let $1\leq q<4$ and set $M_q'=
\max\big(2n,\big\lfloor \frac{2n}{4-q} \big\rfloor +1 \big)$.
Suppose that $m\in L^q 
\cap \mc C^{M_q'}
$ is a function on $\R^{2n}$ satisfying~\eqref{e09182} and~\eqref{e09181} for all $|\al |\le M_q'$, and let $(r_k)_{k\in \mathbb Z}$ be a sequence such that $|r_k|\leq 1$ for every $k\in \mathbb Z$. Then the bilinear operator $T$ given in~\eqref{E:T} has a bounded extension from $L^2 \times L^2\to  L^1$  
that satisfies
\begin{equation}\label{E:boundedness_T}
\|T\|_{L^2\times L^2\to L^1} \le const( C',C_\al, \delta, q, \|m\|_{L^q}), 
\end{equation}
where $C'$, $C_\al$ and $\delta$ are the constants in ~\eqref{e09182} and~\eqref{e09181}. 

Conversely, if $q\geq 1$ and inequality~\eqref{E:boundedness_T} is satisfied for every $m\in L^q$ fulfilling~\eqref{e09182} and~\eqref{e09181} and for every sequence $(r_k)_{k\in \Z}$ with $|r_k|\leq 1$, then we must necessarily have $q<4$. 
\end{thm}

Using standard duality and interpolation arguments, we can show that 
Theorem~\ref{04161} in fact implies the following more general result.

\begin{cor}\label{11245}
For $1\le q<4$ and $M_q'$ as in Theorem~\ref{04161},  let 
$m$ be a function in $ L^q(\R^{2n})  \cap \mc C^{M_q'}(\R^{2n})$  satisfying~\eqref{e09182} and~\eqref{e09181} for all $|\al |\le M_q'$. Assume that $(r_k)_{k\in \mathbb Z}$ is any bounded sequence. Then 
$$
\|T\|_{L^{p_1}\times L^{p_2}\to L^p}<\nf
$$
whenever $2\le p_1,p_2\le\nf$, $1\leq p\leq 2$ and $ 1/p= 1/{p_1}+ 1/{p_2}$.
\end{cor}


\medskip

We note that   operators $T$ of   type   \eqref{E:T} include rough bilinear  singular integrals, in particular those studied in \cite{CM2} and \cite{Grafakos2015}. Indeed, if 
$K(y,z)=\Omega(y,z)/|(y,z)|^{2n}$ is the kernel of a  rough bilinear singular integral, with say $\Omega$ in $L^\infty(\mathbb S^{2n-1}$), and $\psi$ is a smooth function supported in the unit annulus on $\R^{2n}$ satisfying $\sum_{j\in \Z} \psi(2^{-j}\cdot)=1$, then the rough bilinear singular integral operator $T_\Omega$ is an operator of the form \eqref{E:T}
with  $m=\wh{K\psi}$ and  $r_k=1$ for every $k$.

Theorem~\ref{04161} and Corollary~\ref{11245} can also be viewed as bilinear counterparts of the results in Duoandikoetxea and  Rubio~de Francia~\cite{Duoandikoetxea1986}, in particular of Theorem B in that reference. In direct analogy to the linear case studied in~\cite{Duoandikoetxea1986}, our results have immediate applications to the study of boundedness properties of rough bilinear singular integral operators studied in ~\cite{Grafakos2015} and bilinear  dyadic  spherical maximal operators,  whose continuous counterparts  were studied in  ~\cite{GGIPS}
  and
\cite{Barrionuevo2017}. 
These applications are presented  in Section~\ref{S:applications}.

\medskip

In this work we  also study multipliers $m$ all of whose partial derivatives are merely in $L^\nf$; we say  that such   functions lie in $\mc L^\nf$. To make things precise,  we define the space
$$
\mc L^\nf (\mathbb R^{2n})=\big\{ m:\mathbb R^{2n}\to \mathbb C:\,\, 
\p^\al m \,\,\textup{exist for all $\al$ and } \| \p^\al m\|_{L^\infty}<\infty\big\}.
$$
These functions  play an important role in the study of multipliers of the form
\eqref{e02161}.

In the linear setting if $m\in L^\nf$, then the corresponding linear operator is bounded on $L^2$.
One might guess that a bilinear operator $T_m$ is bounded from $L^2\times L^2$ to $L^1$ when
$m$ lies in $\mc L^\nf$. However
B\'enyi and Torres~\cite{BT} provided an important counterexample of a  function $m\in\mc L^\nf$
for which the associated bilinear operator $T_m$ is unbounded from $L^{p_1}
\times L^{p_2}$ to $L^p$ for any $1\le p_1,p_2<\nf$ satisfying $1/p= 1/{p_1}+1/{p_2}$.
The same authors also proved  in~\cite{BT}  that 
if  all  derivatives  $\p^\al m$ have finite $L^2_\xi L^1_\eta$ and $L^2_\eta L^1_\xi$ norms, 
then $T_m$ is bounded from $L^2\times L^2$ to $L^1$.
The counterexample of B\'enyi and Torres is also complemented by a 
subsequent positive result of Honz\'ik and the first two authors~\cite[Corollary 8]{Grafakos2015},  
who showed that the mere
$L^2$ integrability of functions in $\mc L^\nf$ suffices to
yield the $L^2\times L^2\to L^1$ boundedness of $T_m$. This result appeared in connection with the study of bilinear rough singular integrals and,  on one hand it 
 simplifies the mixed norm conditions in~\cite{BT}, on the other hand it eliminates 
any  decay requirement on 
the derivatives of the multiplier $m$, making it more suitable than H\"ormander type conditions in   situations when only boundedness of derivatives of $m$ is  a priori known. 

Next we   investigate    the   magnitude of integrability  of a given 
function $m $ in  $ \mc L^\nf$ in order for the bilinear multiplier operator $T_m$   be bounded from 
$L^2\times L^2 \to L^1$. We   determine the   optimal degree of integrability  
required to ensure the aforementioned boundedness.  We have the following result.


\begin{thm}\label{Thm1.1}
Let  $1 \leq q<4$ and set $M_q=\left\lfloor \frac{2n}{4-q} \right\rfloor +1$. 
Let $m(\xi,\eta)$ be a function in $ L^q(\mathbb R^{2n}) \cap \mc C^{M_q}(\mathbb R^{2n})
$   satisfying  
\begin{equation}\label{98hggr}
\|\p^\alpha m\|_{L^{\infty}}\le C_{0}<\infty \, \q \text{for all multiindices $\al$ with } |\alpha|\leq M_q.
\end{equation}
Then there is a constant $C$ depending on $n$ and $q$ such that the bilinear operator $T_m$ with multiplier $m$ satisfies
\begin{equation}\label{e05162}
\|T_m \|_{L^2\times L^2\to L^1}\le C\, C_0^{1-\tfrac {q}{4}}\|m\|_{L^q}^{\tfrac {q}4}.
\end{equation}
Conversely, there is a function $m\in \bigcap_{q>4} L^q(\R^{2n}) \cap \mc L^{\infty}(\R^{2n})$ 
such that the associated operator $T_m$ does not map $L^2 \times L^2$ to $L^1$.
\end{thm}

As in Corollary~\ref{11245},  duality and interpolation combined with Theorem~\ref{Thm1.1} in fact yields    $L^{p_1}\times L^{p_2} \rightarrow L^p$ boundedness for the operator $T_m$ in a substantially larger set of indices $(p_1,p_2,p)$.

\begin{cor}\label{11241}
Let $1 \leq q<4$ and set $M_q=\left\lfloor \frac{2n}{4-q} \right\rfloor +1$. Assume that $m\in L^q(\mathbb R^{2n}) 
\cap \mc C^{M_q}(\mathbb R^{2n})
$  satisfies \eqref{98hggr}.       
Then 
$$
\|T_m\|_{L^{p_1}\times L^{p_2}\to L^p}<\infty
$$ 
for all indices $(p_1,p_2,p)$ satisfying $2\le p_1,p_2 \le \infty$, 
$1\le p\le 2$,    
and $1/p=1/p_1+1/p_2$.
\end{cor}






The sufficiency directions of our proofs are   based on the 
product-type wavelet method initiated by  Honz\'ik
and the first two authors in \cite{Grafakos2015} but 
incorporate several crucial combinatorial improvements, while 
the necessary directions use constructions inspired by those in  \cite{GraHeHon}.

\section{Preliminary material}\label{S:prelim}

We   recall some facts related to product-type wavelets. 
  For a fixed $k\in \mathbb N$ there  exist real-valued compactly supported functions $\psi_F,\psi_M$ in $\mathcal C^k(\mathbb R)$, 
which satisfy   $\|\psi_F\|_{L^2(\mathbb R)}=\|\psi_M\|_{L^2(\mathbb R)}=1$, 
  for all $0\le\al\le k$   we have $\int_{\mathbb R}x^{\al}\psi_M(x)dx=0$,  
and,  if $\Psi^G$ is defined by 
$$
\Psi^{G}(x\,)=\psi_{G_1}(x_1)\cdots \psi_{G_{2n}}(x_{2n}) 
$$
for   $G=(G_1,\dots, G_{2n})$ in the set     
$$
 \mathcal I :=\Big\{ (G_1,\dots, G_{2n}):\,\, G_i \in \{F,M\}\Big\}  \, , 
$$
then the  family of 
functions 
$$
\bigcup_{\mu \in  \mathbb Z^{2n}}\bigg[  \Big\{   \Psi^{(F,\dots, F)} (x- \mu  )  \Big\} \cup \bigcup_{\la=0}^\nf
\Big\{  2^{\la n}\Psi^{G} (2^{\la}x-\mu):\,\, G\in \mathcal I\setminus \{(F,\dots , F)\}  \Big\}  
  \bigg]
$$
forms an orthonormal basis of $L^2(\mathbb R^{2n})$, where $x= (x_1, \dots , x_{2n})$.  
 The proof of this fact can be found in  Triebel~\cite{Tr1}.  

Let us denote 
by $\mathcal J$ the set of all pairs $(\lambda,G)$ such that either $\lambda=0$ and $G=(F,\dots,F)$, or $\lambda$ is a nonnegative integer and $G\in \mathcal I \setminus \{(F,\dots,F)\}$. 
For $(\lambda,G)\in \mathcal J$ and $\mu\in \Z^{2n}$ we set
  $$
\Psi^{\la,G}_{\mu} (x\, )=  2^{\la n}\Psi^{G} (2^{\la} x -\mu)\, , \qq x\,\in \mathbb R^{2n}. 
  $$ 
The following lemma is crucial in this work.

\begin{lm}\label{05162}
Let $M$ be a positive integer.
Assume that $m\in \mc C^{M+1}$ is a function on $\R^{2n}$  such that
$$
\sup_{|\alpha | \le M+1} \|\p^\alpha m\|_{L^{\infty}}\le C_{0}<\infty \, .
$$
Then we have 
\begin{equation}\label{887}
|\langle \Psi^{\la,G}_{\mu}, m\rangle| \leq C C_02^{-(M+1+n)\la} \, ,
\end{equation}
provided that $\psi_M$ has $M$ vanishing moments.
\end{lm}

This lemma is essentially Lemma 7 in \cite{Grafakos2015}
and its proof is omitted.

\section{Proof of Theorem~\ref{Thm1.1}}


\begin{proof}
 \textit{Sufficiency.}
We utilize the wavelet decomposition of $m$ described in Section~\ref{S:prelim}. We recall that our wavelets are compactly supported and the function $\psi_M$ has $M$ vanishing moments, where $M$ is a certain (large) number to be determined later.  

For $(\lambda,G)\in \mathcal J$ and $\mu\in \Z^{2n}$ we set
$$
b_{\mu}^{\lambda,G}=\langle \Psi_{\mu}^{\lambda,G},m \rangle.
$$
By~\cite[Theorem 1.64]{TrIII} and by the fact that $L^q=F^0_{q,2}$, we obtain
\begin{equation}\label{E:triebel_equivalence}
\|m\|_{L^q}\approx \Big\|\Big(\sum_{(\lambda,G)\in \mathcal J} \sum_{\mu \in \Z^{2n}} |b_{\mu}^{\lambda,G} 2^{\la n}\chi_{Q_{\lambda \mu}}|^2\Big)^{1/2} \Big\|_{L^q},
\end{equation}
where 
$Q_{\lambda \mu}$ is the cube centered at $2^{-\la} \mu$ with sidelength $2^{1-\la}$.

Now, let us fix $(\lambda,G)\in \mathcal J$. To simplify notation, we will write $b_{\mu}$ instead of $b_{\mu}^{\lambda,G}$ in what follows. We also denote by $\tilde{Q}_{\lambda \mu}$ the cube centered at  $2^{-\la}\mu$ with sidelength $2^{-\la}$. Noting that these cubes are pairwise disjoint in $\mu$ (for the fixed value of $\lambda$), the equivalence~\eqref{E:triebel_equivalence} yields
\begin{align}\label{E:cube}
\|m\|_{L^q} 
&\gtrsim 2^{\lambda n} \big\|\big(\sum_{\mu \in \Z^{2n}} |b_\mu|^2 \chi_{Q_{\lambda \mu}} \big)^{\frac{1}{2}}\big\|_{L^q}
\geq 2^{\lambda n} \big\|\big(\sum_{\mu \in \Z^{2n}} |b_\mu|^2 \chi_{\tilde{Q}_{\lambda \mu}} \big)^{\frac{1}{2}}\big\|_{L^q}\\
&=2^{\lambda n} \big\|\sum_{\mu\in \Z^{2n}} |b_\mu| \chi_{\tilde{Q}_{\lambda \mu}}\big\|_{L^q}
=2^{\lambda n(1-\frac{2}{q})} \big(\sum_{\mu\in \Z^{2n}} |b_\mu|^q \big)^{\frac{1}{q}}.
\nonumber
\end{align}

Setting $b= (b_{\mu})_{\mu \in \Z^{2n}}$, \eqref{E:cube} becomes $\|b\|_{\ell^q} \leq C 2^{\lambda n(\frac{2}{q}-1)} \|m\|_{L^q}$. Also, Lemma~\ref{05162} implies that
   $\|b\|_{\ell^\nf}\le CC_02^{-(M+n+1)\la}$,  
where $M$ is the to be determined number of vanishing moments of $\psi_M$.  

For a nonnegative integer $r$ we define 
$$
U_r=\{(k,l) \in \mathbb Z^{n} \times \Z^n=\Z^{2n}:2^{-r-1}\|b\|_{\ell^\infty}<|b_{(k,l)}|\leq 2^{-r}\|b\|_{\ell^\infty}\}.
$$
We can write $U_r$ as a union of the following two disjoint sets:
\begin{align*}
&U_r^1=\{(k,l) \in U_r: {\rm card} \{s: (k,s) \in U_r\} \geq 2^{\frac{rq}{2}} \|b\|_{\ell^q}^{\frac{q}{2}} \|b\|_{\ell^\infty}^{-\frac{q}{2}}\};\\
&U_r^2=\{(k,l) \in U_r: {\rm card} \{s: (k,s) \in U_r\} < 2^{\frac{rq}{2}} \|b\|_{\ell^q}^{\frac{q}{2}} \|b\|_{\ell^\infty}^{-\frac{q}{2}}\}.
\end{align*}
Let us denote
$$
E=\{k \in \Z^n:(k,l)\in U_r^1 \textup{ for some } l \in \Z^n\}.
$$
Then 
$$
{\rm card} E \cdot 2^{\frac{rq}{2}} \|b\|_{\ell^q}^{\frac{q}{2}} \|b\|_{\ell^\infty}^{-\frac{q}{2}} 2^{-q(r+1)} \|b\|_{\ell^\infty}^q \leq \sum_{(k,l)\in U_r^1} |b_{(k,l)}|^q \leq \|b\|_{\ell^q}^q,
$$
and therefore
$$
{\rm card} E \leq C 2^{\frac{rq}{2}} \|b\|_{\ell^q}^{\frac{q}{2}} \|b\|_{\ell^\infty}^{-\frac{q}{2}}\leq C 2^{\frac{rq}{2}} 2^{\lambda n-\frac{\lambda nq}{2}} \|m\|_{L^q}^{\frac{q}{2}} \|b\|_{\ell^\infty}^{-\frac{q}{2}}.
$$

Given $(k,l)\in \Z^{n} \times \Z^n$, it follows from the definition of $\Psi_{(k,l)}^{\lambda,G}$ that $\Psi_{(k,l)}^{\lambda,G}$ can be written in the tensor product form $\Psi_{(k,l)}^{\lambda,G}=\omega_{1,k} \omega_{2,l}$, where $\omega_{1,k}$ is a function of the variables $(x_1,\dots,x_n)$, $\omega_{2,l}$ is a function of $(x_{n+1},\dots,x_{2n})$ and $\|\omega_{1,k}\|_{L^\infty} \approx \|\omega_{2,l}\|_{L^\infty}=2^{\frac{\lambda n}{2}}$. 
Define
$$
m^{r,1}=\sum_{(k,l)\in U^1_r} b_{(k,l)} \Psi_{(k,l)}^{\lambda,G}=\sum_{(k,l)\in U^1_r} b_{(k,l)} \om_{1,k} \om_{2,l}. 
$$
Let ${\mathcal F}^{-1}$ denote the inverse Fourier transform. 
Then 
\begin{align*}
\big\| T_{m^{r,1}}(f,g)\big\|_{L^1}  
\leq &\Big\|\sum_{(k,l)\in U_r^1} b_{(k,l)}{\mathcal F}^{-1}(\omega_{1,k}\widehat f\,){\mathcal F}^{-1} (\omega_{2,l}\widehat g\,)\Big\|_{L^1}\\
\leq\, &\sum_{k\in E} \big\|\omega_{1,k}\widehat f\,\big\|_{L^2} \Big\|\sum_{l: (k,l)\in U_r^1} b_{(k,l)} \omega_{2,l}\widehat g \,\Big\|_{L^2}\\
\le\,&\,\,C\sum_{k\in E} \big\|\omega_{1,k}\wh f\, \big\|_{L^2} 2^{\frac{\la n}{2}}2^{-r}\|b\|_{\ell^\infty} \|g\|_{L^2}\\
\le\,&\,\,C\Big(\sum_{k\in E} 1\Big)^{1/2}\Big(\sum_{k\in E}\big\|\om_{1,k}\wh f\, \big\|_{L^2}^2\Big)^{\frac{1}{2}}2^{\frac{\la n}{2}}
2^{-r}\|b\|_{\ell^\nf}\|g\|_{L^2}\\
\leq\,&\, \, C 2^{\frac{\lambda n}{2} +\lambda n(1-\frac{q}{4})} 2^{\frac{rq}{4}-r} \|m\|_{L^q}^{\frac{q}{4}} \|b\|_{\ell^{\infty}}^{1-\frac{q}{4}} \|f\|_{L^2} \|g\|_{L^2}. 
\end{align*}
Notice that in the estimates above we used the property that the supports of the functions $\omega_{1,k}$ and $\omega_{2,l}$ only have finite overlaps.

Now define
$$
m^{r,2}=\sum_{(k,l)\in U_r^2} b_{(k,l)} \omega_{1,k} \omega_{2,l}.
$$
Then
\begin{align*}
\|T_{m^{r,2}}(f,g)\|_{L^1} 
&\leq \Big\|\sum_{(k,l)\in U_r^2} b_{(k,l)}{\mathcal F}^{-1}(\omega_{1,k}\widehat f\,){\mathcal F}^{-1} (\omega_{2,l}\widehat g\,)\Big\|_{L^1}\\
&\leq \sum_{l\in \mathbb Z^n} \big\|\omega_{2,l}\widehat g\,\big\|_{L^2} \Big\|\sum_{k: (k,l)\in U_r^2} b_{(k,l)} \omega_{1,k}\widehat f \,\Big\|_{L^2}\\
&\leq \Big(\sum_{l\in \mathbb Z^n} \big\|\omega_{2,l}\widehat g\,\big\|_{L^2}^2\Big)^{\frac{1}{2}} 
\Big(\sum_{l\in \mathbb Z^n} \Big\|\sum_{k: (k,l)\in U_r^2} b_{(k,l)} \omega_{1,k}\widehat f \,\Big\|_{L^2}^2 \Big)^{\frac{1}{2}}\\
&\leq C 2^{\frac{\lambda n}{2}} \|g\|_{L^2} \Big(\sum_{k\in \mathbb Z^n} \big\|\omega_{1,k} \widehat f\, \big\|_{L^2}^2 \sum_{l: (k,l)\in U_r^2} |b_{(k,l)}|^2 \Big)^{\frac{1}{2}}\\
&\leq C 2^{\frac{\lambda n}{2}} \|g\|_{L^2} 2^{\frac{rq}{4}-r} \|b\|_{\ell^q}^{\frac{q}{4}} \|b\|_{\ell^\infty}^{1-\frac{q}{4}} \Big(\sum_{k\in \mathbb Z^n} \big\|\omega_{1,k} \widehat f\, \big\|_{L^2}^2 \Big)^{\frac{1}{2}}\\
&\leq C 2^{\frac{\lambda n}{2} +\lambda n(1-\frac{q}{4})} 2^{\frac{rq}{4}-r} \|m\|_{L^q}^{\frac{q}{4}} \|b\|_{\ell^\infty}^{1-\frac{q}{4}} \|f\|_{L^2} \|g\|_{L^2} .
 \end{align*}

Let  
$$
 m^{r }=\sum_{(k,l)\in U_r } b_{(k,l)} \om_{1,k} \om_{2,l}.
 $$
Then the previous estimates yield
$$
\|T_{m^r}\|_{L^2\times L^2\to L^1}
\le C 2^{\frac{\lambda n}{2} +\lambda n(1-\frac{q}{4})} 2^{\frac{rq}{4}-r} \|m\|_{L^q}^{\frac{q}{4}} \|b\|_{\ell^\infty}^{1-\frac{q}{4}}.
$$
Using that 
$\|b\|_{\ell^\nf}\le CC_02^{-\la(M+1+n)}$, we obtain
$$
\|T_{m^r}\|_{L^2\times L^2\to L^1} \leq C C_0^{1-\frac{q}{4}} 2^{-\lambda(\frac{4-q}{4} (M+1)-\frac{n}{2})} 2^{-r(1-\frac{q}{4})} \|m\|_{L^q}^{\frac{q}{4}}.
$$ 
Thus, if we choose $M=\lfloor \frac{2n}{4-q} \rfloor$,
we can sum over $r$ and $(\lambda,G)$ in order, which implies \eqref{e05162}.

\bigskip

 \textit{Necessity.}
We now pass to the converse direction. 
Let $\varphi$ be a Schwartz function on $\R$ whose Fourier transform is supported in the interval $[-1/100,1/100]$, and let $(b_j)_{j=1}^\infty$ and $(d_j)_{j=1}^\infty$ be two sequences of nonnegative numbers with only finitely many nonzero terms. Define functions $f$ and $g$ on $\R^n$ in terms of their Fourier transform by
\begin{equation}\label{E:f}
\wh{f}(\xi)=\sum_{j=1}^\infty b_j \wh{\varphi}(\xi_1-j) \prod_{r=2}^n \wh{\varphi}(\xi_r-1) 
\end{equation}
and
\begin{equation}\label{E:g}
\wh{g}(\eta)=\sum_{k=1}^\infty d_k \wh{\varphi}(\eta_1-k) \prod_{r=2}^n \wh{\varphi}(\eta_r-1),
\end{equation}
with the convention that $\prod_{r=2}^n \wh{\varphi}(\eta_r-1)=1$ if $n=1$.
Then both $f$ and $g$ are Schwartz functions whose $L^2$-norms are bounded by a constant multiple of $\big(\sum_{j=1}^\infty b_j^2\big)^{\frac{1}{2}}$ and $\big(\sum_{k=1}^\infty d_k^2\big)^{\frac{1}{2}}$, respectively. 

Let $(a_j(t))_{j=1}^\infty$ denote the sequence of Rademacher functions; see, for instance, \cite[Appendix C]{CFA}
for the definition and basic properties of Rademacher functions. For any $t\in [0,1]$, consider the function $m_t$ given by
$$
m_t(\xi,\eta)=\sum_{j=1}^\infty \sum_{k=1}^\infty a_{j+k}(t) c_{j+k} \psi(\xi_1-j) \psi(\eta_1-k) \prod_{r=2}^n \psi(\xi_r-1) \psi(\eta_r-1),
$$
where $(c_\ell)_{l=2}^\infty$ is a bounded sequence of nonnegative numbers and $\psi$ is a smooth function on $\R$ supported in the interval $[-1/10,1/10]$ assuming value 1 in $[-1/20,1/20]$. Then
\begin{align*}
&T_{m_t}(f,g)(x)\\
&=\sum_{j=1}^\infty \sum_{k=1}^\infty b_j d_k a_{j+k}(t) c_{j+k} (\varphi(x_1))^2 e^{2\pi ix_1(j+k)} \prod_{r=2}^n e^{4\pi ix_r} (\varphi(x_r))^2 \\
&=\sum_{l=2}^\infty a_l(t) c_l e^{2\pi ix_1 l} (\varphi(x_1))^2 \sum_{j=1}^{l-1} b_j d_{l-j} \prod_{r=2}^n e^{4\pi ix_r} (\varphi(x_r))^2.
\end{align*}

Using Fubini's theorem and Khintchine's inequality (see, for instance, \cite[Appendix C]{CFA}), we obtain
\begin{align}\label{E:l1}
&\int_0^1 \|T_{m_t}(f,g)\|_{L^1} \,dt\\
\nonumber
&= \left(\int_{\mathbb R} |\varphi(y)|^2\,dy\right)^{n-1}\int_{\mathbb R} \int_0^1 \left|\sum_{l=2}^\infty a_l(t) c_l e^{2\pi ix_1 l} (\varphi(x_1))^2 \sum_{j=1}^{l-1} b_j d_{l-j}\right|\,dx_1\,dt\\
\nonumber
&\approx \left(\int_{\mathbb R} |\varphi(y)|^2\,dy\right)^{n-1} \int_{\mathbb R} \left(\sum_{l=2}^\infty \left(c_l |\varphi(x_1)|^2 \sum_{j=1}^{l-1} b_j d_{l-j}\right)^2 \right)^{\frac{1}{2}}\,dx_1\\
\nonumber
&=\left(\sum_{l=2}^\infty c_l^2\left(\sum_{j=1}^{l-1} b_j d_{l-j}\right)^2 \right)^{\frac{1}{2}} \left(\int_{\mathbb R} |\varphi(y)|^2\,dy\right)^{n}. 
\end{align}

We now fix a positive integer $N\geq 2$ and set $b_j^N=d_j^N=2^{-\frac{N}{2}}$ if $j=2^N, \dots, 2^{N+1}-1$ and $b_j^N=d_j^N=0$ otherwise. Then the sequences $(b_j^N)_{j=1}^\infty$ and $(d_j^N)_{j=1}^\infty$ belong to the unit ball in $\ell^2$. We also observe that if $j\in \mathbb N$ and $l<2^{N+1}$ then either $j<2^N$ or $l-j<2^N$, and if $l>2^{N+2}-2$ then either $j>2^{N+1}-1$ or $l-j>2^{N+1}-1$. Consequently,
\begin{equation}\label{E:zero}
\sum_{j=1}^{l-1} b_j^N d_{l-j}^N=0 \quad \textup{ if $l<2^{N+1}$ or $l>2^{N+2}-2$}.
\end{equation}
On the other hand, if $5 \cdot 2^{N-1} \leq l \leq 6\cdot 2^{N-1}$ then
\begin{equation}\label{E:nonzero}
\sum_{j=1}^{l-1} b_j^N d_{l-j}^N
\geq \sum_{j=\max\{2^N,l+1-2^{N+1}\}}^{\min\{2^{N+1}-1,l-2^N\}} 2^{-\frac{N}{2}} 2^{-\frac{N}{2}} 
\geq \sum_{j=2\cdot 2^{N-1}+1}^{3 \cdot 2^{N-1}} 2^{-N}\geq \frac{1}{2}.
\end{equation}

We define $c_l=(l-1)^{-\frac{1}{2}} (\log e(l-1))^{\frac{1}{2}}$. If $f^N$ and $g^N$ are given by~\eqref{E:f} and~\eqref{E:g}, respectively, with $b_j$ replaced by $b_j^N$ and $d_k$ replaced by $d_k^N$, then a combination of~\eqref{E:l1} and~\eqref{E:nonzero} yields
\begin{align*}
\int_0^1 &\|T_{m_t}(f^N,g^N)\|_{L^1}\,dt\\ 
&\gtrsim \left(\sum_{l=5\cdot 2^{N-1}}^{6 \cdot 2^{N-1}} c_l^2\right)^{\frac{1}{2}} 
=\left(\sum_{l=5\cdot 2^{N-1}}^{6\cdot 2^{N-1}} (l-1)^{-1} \log e(l-1)\right)^{\frac{1}{2}}\\
&\approx \left(\log^2 6 \cdot 2^{N-1}-\log^2 5\cdot 2^{N-1}\right)^{\frac{1}{2}}\\
&=\left(\log 6\cdot 2^{N-1}+\log 5\cdot 2^{N-1}\right)^{\frac{1}{2}}\left(\log 6\cdot 2^{N-1}-\log 5\cdot 2^{N-1}\right)^{\frac{1}{2}}\\
&\approx N^{\frac{1}{2}}. 
\end{align*} 
Consequently, for every $N\geq 2$ we can find $t_N\in [0,1]$ such that 
\begin{equation}\label{ee02031}
\|T_{m_{t_N}}(f^N,g^N)\|_{L^1} \geq CN^{\frac{1}{2}},
\end{equation}
with $C$ independent of $N$.

Let us now consider the function
$$
m(\xi,\eta)=\sum_{j=1}^\infty \sum_{k=1}^\infty s_{j+k} c_{j+k} \psi(\xi_1-j) \psi(\eta_1-k) \prod_{r=2}^n \psi(\xi_r-1) \psi(\eta_r-1),
$$
where $(s_l)_{l=2}^\infty$ is a sequence taking values in $\{-1,1\}$ and satisfying $s_l=a_l(t_N)$ if $N\geq 2$ and $2^{N+1}\leq l \leq 2^{N+2}-1$. By~\eqref{E:zero} we therefore obtain 
\begin{align*}
T_m&(f^N,g^N)(x)\\
&=\sum_{l={2^{N+1}}}^{2^{N+2}-1} s_l c_l e^{2\pi ix_1 l} (\varphi(x_1))^2 \sum_{j=1}^{l-1} b_j^N d_{l-j}^N \prod_{r=2}^n e^{4\pi ix_r} (\varphi(x_r))^2\\
&=\sum_{l={2^{N+1}}}^{2^{N+2}-1} a_l(t_N) c_l e^{2\pi ix_1 l} (\varphi(x_1))^2 \sum_{j=1}^{l-1} b_j^N d_{l-j}^N \prod_{r=2}^n e^{4\pi ix_r} (\varphi(x_r))^2\\
&=T_{m_{t_N}}(f^N,g^N)(x),
\end{align*} 
which yields
$$
\|T_{m}(f^N,g^N)\|_{L^1} \geq CN^{\frac{1}{2}},
$$
and so $T_m$ is unbounded from $L^2 \times L^2$ into $L^1$. We can also observe that $m$ is a smooth function with all derivatives bounded and
$$
\|m\|_{L^q} \approx \left(\sum_{l=2}^\infty c_l^q (l-1) \right)^{\frac{1}{q}} 
=\left(\sum_{l=2}^\infty (l-1)^{1-\frac{q}{2}} (\log e(l-1))^{\frac{q}{2}}\right)^{\frac{1}{q}},
$$
which implies that $m\in \bigcap_{q>4} L^q$.
\epf

\begin{rmk}
A result related to Theorem~\ref{Thm1.1} appeared in~\cite[Remark 2]{GraHeHon}  asserting that the inequality
\begin{equation}\label{e12221}
\|T_m\|_{L^2\times L^2\to L^1}\le C\|m\|_{L^q_s}
\end{equation}
holds whenever $qs>2n$ and $1<q<4$; here, $L^q_s$ stands for the inhomogeneous Sobolev space with integrability index $q$ and smoothness index $s$. We     observe that the testing functions from the proof of Theorem~\ref{Thm1.1} show that inequality~\eqref{e12221} fails for $q>4$. 
\end{rmk}

\begin{rmk}
The quantities on the right-hand side of inequalities~\eqref{e12221} and~\eqref{e05162} can be compared in certain situations via the classical Gagliardo-Nirenberg inequality~\cite{Gagliardo,Nirenberg}. For instance, if $n$ is an odd integer then
$$
\|m\|_{L^4_{\frac{n+1}{2}}} \leq C\,
 \Big(\sup_{|\alpha|\leq n+1} \|\partial^\alpha m\|_{L^\infty}\Big)^{\frac{1}{2}} \|m\|_{L^2}^{\frac{1}{2}};
$$
we point out, however, that the value $q=4$ is not allowed in~\cite[Remark 2]{GraHeHon}.
\end{rmk}

\section{Proof  of Theorem~\ref{04161}}

Next we prove Theorem~\ref{04161}.  The sufficiency part follows from Theorem~\ref{Thm1.1} via an argument as in~\cite{Grafakos2015}; we provide a  sketch of the proof for the sake of completeness.

\bpf[Proof  of Theorem~\ref{04161}]

\textit{Sufficiency.}
Let $\psi$ be a smooth function supported in the unit annulus such that $\sum_{j\in \Z} \psi(2^{-j}\cdot)=1$. For every $j\in \Z$ we denote $\psi_j=\psi(2^{-j}\cdot)$ and $m_{j,0}=m\psi_j$. Further, for every $k\in \Z$ we set $m_{j,k}(\cdot)=m_{j,0}(2^k\cdot)$. If we define $\widetilde T_j$ to be the bilinear operator with multiplier $\sum_{k \in \Z}
r_km_{j,k}$, then we have
$T=\sum_{j\in \Z} \widetilde T_j$.

The boundedness of $\sum_{j\le 0}\widetilde T_j$ follows from
the bilinear Coifman-Meyer theory, see \cite[Proposition~3]{Grafakos2015} for a detailed argument. We point out that the validity of the estimates~\eqref{e09182} and~\eqref{e09181} for $(\xi,\eta)$ near the origin is essential here, and that this argument requires the restriction $M_q'\geq 2n$.  

Next we sketch the proof when $j>0$. The proof follows the argument in~\cite{Grafakos2015}, with modifications based on Theorem~\ref{Thm1.1}. The goal is to obtain the $L^2\times 
L^2\to L^1$ boundedness of $\widetilde{T_j}$, with bounds forming a convergent series in $j$. We first decompose the multiplier $m_{j,0}$ into its diagonal and off-diagonal parts. 
Precisely, let $m_{j,0}=\sum_\om a_\om\om$ be the wavelet decomposition of $m_{j,0}$, and let $D_1$ denote the collection of wavelets
whose supports have a nonempty intersection with the set $\{(\xi,\eta):\,2^{-j}|\xi|\le|\eta|\le 2^j|\xi|\}$. Then we define the diagonal part
$m_{j,0}^1=\sum_{\om\in D_1}a_\om\om$, and we 
denote by $m_{j,0}^2=m_{j,0}-m_{j,0}^1$ the off-diagonal part of $m_{j,0}$.
We remark that Theorem~\ref{Thm1.1}   yields an estimate for the norm of the bilinear multiplier operator corresponding to $m_{j,0}^1$ due to the fact that 
the proof is only based on 
   estimates of the coefficients $a_\om$.
Using this and a standard dilation argument we obtain the $L^2\times 
L^2\to L^1$ boundedness of the operator  $\widetilde T_j^1$ associated with
the bilinear multiplier $\sum_k r_k m_{j,0}^1(2^k\cdot)$
with   bound 
$$
C\, j\|m_{j,0}\|_{L^q}^{\tf q4}C_0^{1-\tf q4}\le C'\, \|m\|_{L^q}^{\frac{q}{4}} j2^{-j\de(1-\frac{q}{4})},
$$
where 
$$
C_0=\max_{|\al|\le \lfloor \frac{2n}{4-q}\rfloor +1}\|\p^\al m_{j,0}\|_{L^\nf}\le C''\, 2^{-j\de},
$$
by~\eqref{e09181}.
We omit the details which can be obtained by a straightforward 
modification of  the proof in \cite[Section 4]{Grafakos2015}
 with the help of Theorem~\ref{Thm1.1}.

Boundedness of the off-diagonal parts of the operator $\widetilde{T_j}$ can be proved by an argument as in~\cite[Section 5]{Grafakos2015}. We note that this argument relies only on 
the decay of the $L^\nf$-norms of the derivatives of $m_{j,0}$,
which is guaranteed by condition~\eqref{e09181}; this argument requires the restriction
$M_q'> n$.

\medskip

\noindent
\textit{Necessity.}
Throughout the proof we shall adopt the convention that whenever $x\in \R^n$ then $x_r$ denotes the $r$-th coordinate of $x$, that is, $x=(x_1,\dots,x_n)$.

Let $\psi$ be a nonzero Schwartz function on $\R$ whose Fourier transform is supported in $[-1/10,1/10]$, and let $(a_l(t))_{l\in \Z^n}$ denote the sequence of Rademacher functions indexed by the elements of the set $\Z^n$. Given $N\in \mathbb N$, $N\geq 4$ we introduce the following sets: 
\begin{align*}
&I_N=\{5\cdot 2^{N-2}+1,\dots,6\cdot 2^{N-2}-1\}^n;\\
&J_N=\{5\cdot 2^{N-1}+2,\dots,6\cdot 2^{N-1}-2\}^n;\\ 
&L_N=\{41\cdot 2^{N-4}+1, \dots, 43\cdot 2^{N-4}\}^n.
\end{align*}
We observe that $L_N \subseteq J_N = I_N+I_N$. 
For $t\in [0,1]$ and $(\xi,\eta)\in \R^n \times \R^n$ we define 
$$
m_t(\xi,\eta)=\sum_{N=4}^\infty \sum_{j\in I_N} \sum_{k\in I_N} a_{j+k}(t) \prod_{r=1}^n c_{j_r+k_r} \wh{\psi}(\xi_r-j_r) \wh{\psi}(\eta_r-k_r),
$$
where $c_l=l^{-1/2} (\log l)^{-1/n}$, $l\in \mathbb N$, $l\geq 42$. It is straightforward to observe that the family of functions $m_t$ satisfies the estimates~\eqref{e09182} and~\eqref{e09181} with constants independent of $t$, and that
\begin{align*}
\|m_t\|_{L^q}
&\lesssim \left(\sum_{N=4}^\infty \sum_{j\in I_N} \sum_{k\in I_N} \prod_{r=1}^n c_{j_r+k_r}^q\right)^{\frac{1}{q}}
\lesssim \left(\sum_{N=4}^\infty \sum_{l\in J_N} \prod_{r=1}^n c_{l_r}^q 2^{nN}\right)^{\frac{1}{q}}\\
&=\left(\sum_{N=4}^\infty 2^{nN} \left(\sum_{p=5\cdot 2^{N-1}+2}^{6\cdot 2^{N-1}-2} c_p^q\right)^n \right)^{\frac{1}{q}}\\
&\lesssim \left(\sum_{N=4}^\infty 2^{nN(2-\frac{q}{2})}N^{-q}\right)^{\frac{1}{q}}<\infty
\end{align*}
whenever $q\geq 4$. 

Let $f=g$ be Schwartz functions on $\R$ such that $\wh{f}=\wh{g}$ is supported in the interval $(1,2)$ and $\wh{f}=\wh{g}=1$ inside the interval $[5/4,3/2]$, and let $F(x)=G(x)=\prod_{r=1}^n f(x_r)=\prod_{r=1}^n g(x_r)$.
Let $K\in \Z$, $K\geq 4$. Then 
\begin{align*}
&m_t(2^K\xi,2^K\eta) \wh{F}(\xi) \wh{G}(\eta)\\
&=\sum_{N=4}^\infty \sum_{j\in I_N} \sum_{k\in I_N} a_{j+k}(t) \prod_{r=1}^n c_{j_r+k_r} \wh{\psi}(2^K\xi_r-j_r) \wh{\psi}(2^K\eta_r-k_r)\wh{f}(\xi_r) \wh{g}(\eta_r).
\end{align*}
Assume that $r\in \{1,\dots,n\}$ and   $j_r\in\{5\cdot 2^{N-2}+1,\dots, 6\cdot 2^{N-2}-1\}$ (for some $N\geq 4$) are such that the function $\wh{\psi}(2^K\xi_r-j_r)\wh{f}(\xi_r)$ is not 
identically equal to $0$. Using the support properties of $\wh{\psi}$ and $\wh{f}$ we deduce that there is $\xi_r\in (1,2)$ such that $|2^K\xi_r-j_r|\leq 1/10$. Thus,
$$
2^K -\frac{1}{10} <2^K\xi_r -\frac{1}{10} \leq j_r \leq 2^K \xi_r +\frac{1}{10} <2^{K+1}+\frac{1}{10}. 
$$
Consequently, $2^K\leq j_r \leq 2^{K+1}$, which in turn implies that $j_r$ belongs to the set $\{5\cdot 2^{K-2}+1,\dots, 6\cdot 2^{K-2}-1\}$. 

Now, if $j_r$ is as in the previous paragraph and $\xi_r$ satisfies $|2^K\xi_r-j_r|\leq 1/10$ then
$$
\frac{5}{4}\leq \frac{5}{4}+\frac{9}{10\cdot 2^K} \leq \frac{j_r}{2^K}-\frac{1}{10\cdot 2^K} \leq \xi_r 
\leq \frac{j_r}{2^K}+\frac{1}{10\cdot 2^K} \leq \frac{6}{4}-\frac{9}{10 \cdot 2^K} \leq \frac{3}{2},
$$
and so $\wh{f}(\xi_r)=1$. 

The previous observations applied to both $\xi$ and $\eta$ yield
\begin{align}\label{E:m_t}
&m_t(2^K\xi,2^K\eta) \wh{F}(\xi) \wh{G}(\eta)\\
&=\sum_{j\in I_K} \sum_{k\in I_K} a_{j+k}(t) \prod_{r=1}^n c_{j_r+k_r} \wh{\psi}(2^K\xi_r-j_r) \wh{\psi}(2^K\eta_r-k_r). \nonumber
\end{align}

Let $S$ be a finite subset of $\{K\in \Z: K\geq 4\}$. We denote
$$
T^S_t=\sum_{K\in S} T_{m_t(2^K\cdot)}.
$$
Then, by~\eqref{E:m_t}, we obtain
\begin{align*}
&T^S_t(F,G)(x)\\
&=\sum_{K\in S} \sum_{j\in I_K} \sum_{k\in I_K} a_{j+k}(t) \prod_{r=1}^n c_{j_r+k_r} \frac{1}{2^{2K}} \left(\psi\left(\frac{x_r}{2^K}\right)\right)^2 e^{2\pi i x_r\cdot \frac{j_r+k_r}{2^K}}\\
&=\sum_{K\in S} \sum_{l\in J_K} a_l(t) \prod_{r=1}^n c_{l_r} \frac{1}{2^{2K}} \left(\psi\left(\frac{x_r}{2^K}\right)\right)^2 e^{2\pi i x_r\cdot \frac{l_r}{2^K}}\\ 
&\hspace{3cm} \times \min\{l_r-5\cdot 2^{K-1}-1, 6\cdot 2^{K-1}-1-l_r\}.
\end{align*}
We notice that the sets $J_K$ are pairwise disjoint in $K$. 
By Fubini's theorem and Khintchine's inequality we write
\begin{align*}
&\int_0^1 \|T^S_t(F,G)\|_{L^1}\,dt
=\int_{\R^n} \int_0^1 |T^S_t(F,G)(x)|\,dt \,dx\\
&\approx \int_{\R^n} \Biggr(\sum_{K\in S} \sum_{l\in J_K} \prod_{r=1}^n \frac{1}{2^{4K}} \left|\psi\left(\frac{x_r}{2^K}\right)\right|^4 c_{l_r}^2\\ 
&\hspace{3cm}\times \min\{l_r-5\cdot 2^{K-1}-1, 6\cdot 2^{K-1}-1-l_r\}^2\Biggr)^{\frac{1}{2}}\,dx\\
&\gtrsim \int_{\R^n} \left(\sum_{K\in S} \sum_{l\in L_K} \prod_{r=1}^n \frac{1}{2^{2K}} \left|\psi\left(\frac{x_r}{2^K}\right)\right|^4 c_{l_r}^2 \right)^{\frac{1}{2}}\,dx\\
&=\int_{\R^n} \left(\sum_{K\in S} \prod_{r=1}^n \frac{1}{2^{2K}} \left|\psi\left(\frac{x_r}{2^K}\right)\right|^4 \sum_{l\in L_K} \prod_{r=1}^n c_{l_r}^2\right)^{\frac{1}{2}}\,dx\\
&=\int_{\R^n} \left(\sum_{K\in S} \left(\sum_{p=41\cdot 2^{K-4}+1}^{43\cdot 2^{K-4}}  c_{p}^2 \right)^n \prod_{r=1}^n \frac{1}{2^{2K}} \left|\psi\left(\frac{x_r}{2^K}\right)\right|^4  \right)^{\frac{1}{2}}\,dx.
\end{align*}
Since $\psi$ is not constantly equal to $0$, there is $A>0$ such that 
$$
\int_{\{y\in \R:~ A\leq y <2A\}} |\psi(y)|^2\,dy>0.
$$
Noticing that the sets $\{x\in \R: ~A\leq \frac{x}{2^K}<2A\}$ are pairwise disjoint in $K$, we can estimate
\begin{align*}
&\int_0^1 \|T^S_t(F,G)\|_{L^1}\,dt\\
&\gtrsim \int_{\R^n} \left(\sum_{K\in S} \left(\sum_{p=41\cdot 2^{K-4}+1}^{43\cdot 2^{K-4}} c_p^2 \right)^n \prod_{r=1}^n \frac{1}{2^{2K}} \left|\psi\left(\frac{x_r}{2^K}\right)\right|^4 \chi_{\{A\leq \frac{x_r}{2^K}<2A\}}(x_r) \right)^{\frac{1}{2}}\,dx\\
&=\sum_{K\in S} \left(\sum_{p=41\cdot 2^{K-4}+1}^{43\cdot 2^{K-4}} c_p^2 \right)^{\frac{n}{2}} \prod_{r=1}^n \int_{\{x_r: ~A\leq \frac{x_r}{2^K}<2A\}} \frac{1}{2^K} \left|\psi\left(\frac{x_r}{2^K}\right)\right|^2 \,dx_r\\
&=\sum_{K\in S} \left(\sum_{p=41\cdot 2^{K-4}+1}^{43\cdot 2^{K-4}} c_p^2 \right)^{\frac{n}{2}} \left(\int_{\{y:~ A\leq y <2A\}} |\psi(y)|^2\,dy\right)^n\\
&\approx \sum_{K\in S} \left(\sum_{p=41\cdot 2^{K-4}+1}^{43\cdot 2^{K-4}} c_p^2 \right)^{\frac{n}{2}}.
\end{align*}

We have 
$$
\sum_{p=41\cdot 2^{K-4}+1}^{43\cdot 2^{K-4}} c_p^2
=\sum_{p=41\cdot 2^{K-4}+1}^{43\cdot 2^{K-4}} p^{-1} (\log p)^{-\frac{2}{n}}
\gtrsim K^{-\frac{2}{n}}.
$$
Thus, taking $S=\{4,5,\dots,D\}$ for large $D$, we obtain
$$
\int_0^1 \|T^S_t(F,G)\|_{L^1}\,dt
\gtrsim \sum_{K=4}^D \left(\sum_{p=41\cdot 2^{K-4}+1}^{43\cdot 2^{K-4}} c_p^2 \right)^{\frac{n}{2}}
\gtrsim \sum_{K=4}^D K^{-1},
$$
which tends to $\infty$ as $D\to \infty$. This contradicts~\eqref{E:boundedness_T}.
\end{proof}

\section{Consequences and   Corollaries}\label{S:corollaries}

The following result,  first proved in
\cite{Grafakos2015}, provided the inspiration for the work in this article.

\begin{prop}[{\cite[Corollary 8]{Grafakos2015}}]\label{09131}
Suppose that $m(\xi,\eta)$ is a function in $ L^2(\mathbb R^{2n}) \cap \mc C^{4n+1}(\R^{2n})$ such that 
$$
\sup_{|\alpha | \le 4n+1} \|\p^\alpha m\|_{L^{\infty}}\le C_{0}<\infty \, .
$$
Then there is a dimensional constant $C(n)$ such that the bilinear operator $T_m$ with multiplier $m$ satisfies
\begin{equation}\label{e05161}
\|T_m \|_{L^2\times L^2\to L^1}\le C(n)\, C_0^{\frac 1 5}\|m\|_{L^2}^{\frac 45}.
\end{equation}
\end{prop}

\medskip
Theorem~\ref{Thm1.1} allows us to significantly sharpen the preceding result.

\begin{cor}\label{05161}
Suppose that $m(\xi,\eta)$ is a function in $ L^2(\mathbb R^{2n}) \cap \mc C^{n+1}(\R^{2n})$ such that 
$$
\sup_{|\alpha | \le n+1} \|\p^\alpha m\|_{L^{\infty}}\le C_{0}<\infty \, .
$$
Then there is a dimensional constant $C(n)$ such that the bilinear operator $T_m$ with multiplier $m$ satisfies
\begin{equation}\label{e05161*}
\|T_m \|_{L^2\times L^2\to L^1}\le C(n)\, C_0^{\frac 1 2}\|m\|_{L^2}^{\frac 12}.
\end{equation}
Moreover, the exponent  $1/2$ is sharp in~\eqref{e05161*},
in the sense that whenever $K$ is a fixed positive integer and
\begin{equation}\label{E:r}
\|T_m \|_{L^2\times L^2\to L^1}\le C \|m\|_{L^2}^{r}
\end{equation}
holds for some $r>0$ and for all   $m\in \mc C^K(\R^{2n})$ satisfying
$$
\sup_{|\alpha | \le K} \|\p^\alpha m\|_{L^{\infty}}\le 1,
$$
then $r\geq 1/2$.
\end{cor}

 \begin{proof}
Estimate \eqref{e05161*} follows from \eqref{e05162} with $q=2$. 
The sharpness can be seen by an argument similar to the necessity part of the 
 proof of Theorem~\ref{Thm1.1}. Namely, it follows from the proof of Theorem~\ref{Thm1.1} that~\eqref{E:r} implies the validity of the inequality
\begin{equation}\label{E:r2}
\sum_{l=2}^\infty c_l^2 \left(\sum_{j=1}^{l-1} b_j d_{l-j}\right)^2 
\lesssim \left(\sum_{l=2}^\infty c_l^2 (l-1) \right)^r \left(\sum_{j=1}^\infty b_j^2 \right)
\left(\sum_{k=1}^\infty d_k^2 \right)
\end{equation}
for all sequences $(b_j)_{j=1}^\infty$, $(d_k)_{k=1}^\infty$ and $(c_l)_{l=2}^\infty$ of nonnegative numbers such that  $c_l\leq \varepsilon$ for all $l$, where $\varepsilon$ is a certain positive real number depending on $n$ and $K$. Now, if we fix $N\in \mathbb N$ and choose 
$$
c_2=\dots=c_N=\varepsilon , \q  b_1=\dots=b_N=d_1=\dots=d_N=N^{-\frac{1}{2}}
$$
and 
$$
c_i=b_i=d_i=0  \q\textup{if}\q  i>N,
$$
then~\eqref{E:r2} becomes
$$
N\lesssim N^{2r}.
$$
This implies that $r\ge 1/2$.
\end{proof}

\begin{rmk}
We would like to point out that the validity of the argument in the necessity part of the proof of Corollary~\ref{05161} is not limited to the specific value of $q=2$. Indeed, the proof can be easily modified to show the sharpness of the exponent $q/4$ in inequality~\eqref{e05162} for any $q\in [1,4)$.

The argument above can  also be applied in the borderline case $q=4$, showing that, for any fixed positive integer $K$, there is \textit{no} positive constant $\varepsilon$ for which
$$
\|T_m\|_{L^2\times L^2\to L^1}\le CC_0^{\varepsilon} \|m\|_{L^4}^{1-\varepsilon},
$$
where
$$
C_0=\sup_{|\alpha | \le K} \|\p^\alpha m\|_{L^{\infty}}.
$$
It is unknown to us whether $T_m$ is bounded from $L^2 \times L^2 \to L^1$ for every $m\in L^4 \cap \mc L^\infty$; it seems, however, that the decay of the bound with $C_0$, rather than the mere fact that $T_m$ is bounded, is what is relevant in many applications; see Theorem~\ref{04161}. 
\end{rmk}

We complement Corollary~\ref{05161} by showing that the requirement on the number  $M_q=\left\lfloor \frac{2n}{4-q} \right\rfloor +1$ of derivatives  in Theorem~\ref{Thm1.1} is optimal as well.

\begin{prop}
Let $1\le q<4$. Let $M=M(q)$ be an integer such that for all
$m$ in $ L^q(\mathbb R^{2n}) \cap \mc C^{M }(\mathbb R^{2n})
$   satisfying  
\begin{equation}
\|\p^\alpha m\|_{L^{\infty}}\le C_{0}<\infty \, \q \text{for all multiindices $\al$ with } |\alpha|\leq M
\end{equation}
there is a constant $C$ depending on $n$ and $q$ such that the bilinear operator $T_m$ with multiplier $m$ satisfies
\begin{equation}\label{ee05162}
\|T_m \|_{L^2\times L^2\to L^1}\le C\, C_0^{1-\tfrac {q}{4}}\|m\|_{L^q}^{\tfrac {q}4}.
\end{equation}
Then $M\ge \frac{2n}{4-q} $.
\end{prop}

\bpf
The proof is rather straightforward. We take 
$$
m(\xi,\eta)=\prod_{j=1}^n\psi(\xi_j)\psi(\eta_j)
$$ 
 and $f_\la(x)=g_\la(x)=2^{-\la/2}\prod_{j=1}^n\vp(2^{-\la}x_j)$,
where $\psi$ and $\vp$ are the functions from the necessity part of the proof of Theorem~\ref{Thm1.1}.
Let $m_\la(\xi,\eta)=m(2^\la\xi,2^\la\eta)$. Then 
$\|T_{m_\la}(f_\la,g_\la)\|_{L^1}\sim 1$,
while $\|m_\la\|_{L^q}\sim 2^{-2n\la/q}$, and $C_0 \sim 2^{\la M}$. So \eqref{ee05162} implies the inequality
$$
1\le C2^{\la M(1-\tf q4)-\tf{n\la}{2}}.
$$
This shows that $M\ge \frac{2n}{4-q}$ by letting $\la\to\nf$.
\epf



Next, we turn to the proofs of the claimed corollaries in Section 1.

\bpf[Proof of Corollary~\ref{11241}]
Consider the  multipliers
\begin{equation}\label{defmstar}
m_1(\xi,\eta)=m(-(\xi+\eta),\eta), \qq m_2(\xi,\eta)=m(\xi,-(\xi+\eta))
\end{equation}    
of the two adjoints of $T_m$, 
$(T_m)^{*1}$ and $(T_m)^{*2}$.
It is straightforward to verify that   $m_1$ and $m_2$ belong to $L^q \cap \mathcal C^{M_q}$, with the $L^q$-norms of $m_1$ and $m_2$ being comparable to the $L^q$ norm of $m$, and that
$$
\sup_{|\al |\le M_q}\|\p^\al m_i\|_{L^\nf}\le CC_0 , \qq i=1,2 .
$$
 Therefore, by Theorem~\ref{Thm1.1} we have 
$$ 
\|T_{m_1}\|_{L^2\times L^2\to L^1},\, \|T_{m_2}\|_{L^2\times L^2\to L^1}
\le C\, C_0^{1-\tfrac {q}{4}}\|m\|_{L^q}^{\tfrac {q}4},
$$
which, by duality, implies that
\begin{equation}\label{E:endpoints}
\|T_{m}\|_{L^\nf\times L^2\to L^2},\, \|T_{m}\|_{L^2\times L^\nf\to L^2}
\le C\, C_0^{1-\tfrac {q}{4}}\|m\|_{L^q}^{\tfrac {q}4}.
\end{equation}
Interpolating between the estimates~\eqref{E:endpoints} and the estimate~\eqref{e05162} from Theorem~\ref{Thm1.1}, we deduce via~\cite[Corollary 7.2.11]{MFA} that
$$
\|T_{m}\|_{L^{p_1}\times L^{p_2}\to L^p}
\le C\, C_0^{1-\tfrac {q}{4}}\|m\|_{L^q}^{\tfrac {q}4}
$$
whenever $2\le p_1,p_2\le\nf$, $1\leq p\leq 2$ and $ 1/p= 1/{p_1}+ 1/{p_2}$.
\epf

\bpf[Proof of Corollary~\ref{11245}]
Let us first observe that the functions $m_1$ and $m_2$ defined in \eqref{defmstar} satisfy estimates
\eqref{e09182} and \eqref{e09181}. We only verify~\eqref{e09182} for $m_1$ here as the remaining inequalities can be proved similarly.
Notice that $|\xi|+|\eta|\le 2(|\xi+\eta|+|\eta|)$ and $|\xi+\eta|+|\eta|\le 2(|\xi|+|\eta|)$, so \eqref{e09182}  for $m$
implies that 
\begin{align*}
|m_1(\xi,\eta)|=\,&\,|m(-(\xi+\eta),\eta)|\\
\,\le&\, C'\min\big(|(-(\xi+\eta),\eta)|, |(-(\xi+\eta),\eta)|^{-\de}\big)\\
\,\le&\, C\min\big(|(\xi,\eta)|, |(\xi,\eta)|^{-\de}\big),
\end{align*}
which indeed proves \eqref{e09182} for $m_1$. Using this and arguing as in the proof of Corollary~\ref{11241},
we deduce the conclusion.
\epf

\section{Applications}\label{S:applications}

\subsection{Rough bilinear  singular integrals}

Let $\Om$ be a function in $L^r(\mathbb S^{2n-1})$ for some $r>1$ with vanishing integral. We denote $(y,z)'=(y,z)/|(y,z)| \in \mathbb S^{2n-1}$ and
define the rough bilinear singular integral operator $T_\Om$ by
$$
T_\Om(f,g)(x)=\textup{p.v.} \int_{\bbr^{2n}}\f{\Om((y,z)')}{|(y,z)|^{2n}}f(x-y)g(x-z)dydz.
$$ 
This operator was introduced and first studied by Coifman and Meyer~\cite{CM2}. In  \cite{Grafakos2015} 
it was proved that  $T_\Om$ is bounded 
from $L^2\times L^2$ to $L^1$ provided that $\Om\in L^r(\mathbb S^{2n-1})$
for $r\ge2$. As a consequence of  Corollary~\ref{11245} 
we  can improve this result, answering partially question (b) raised in \cite{Grafakos2015}, as follows:

\begin{thm}\label{11244}
Let $r> 4/3$, and assume that $\Om\in L^r(\mathbb S^{2n-1})$ with $\int_{\mathbb S^{2n-1}}  \Om\, d\sigma=0$.
Then we have 
$$
\|T_\Om\|_{L^{p_1}\times L^{p_2}\to L^p}<\nf
$$
whenever $2\leq p_1, p_2 \leq \infty$, $1\leq p\leq 2$ and $\frac{1}{p}=\frac{1}{p_1}+\frac{1}{p_2}$.
\end{thm}
\bpf
We denote $K^0(x)=\tf{\Om(x')}{|x|^{2n}}\psi(x)$, where $\psi$ is a smooth function supported in the unit annulus 
of $\R^{2n}$ satisfying $\sum_{k\in \Z} \psi(2^{-k}\cdot)=1$, and set $m=\wh{K^0}$. It is well known that $m$ satisfies conditions \eqref{e09182} and \eqref{e09181} (see, e.g., \cite[Lemma 8.20]{Duoandikoetxea2000}). Thanks to the embedding $L^{r_1}(\mathbb S^{2n-1}) \subseteq L^{r_2}(\mathbb S^{2n-1})$ if $r_1 \geq r_2$, we may assume that $r\leq 2$. Then, by the Hausdorff-Young inequality, we obtain
\begin{equation}\label{e09011}
\|m\|_{L^{r'}} \leq C\|K^0\|_{L^r} \leq C \|\Omega\|_{L^r}.
\end{equation}
Since $r>4/3$, we have $r'<4$, and Corollary~\ref{11245} applied with $r_k=1$ for all $k$ thus yields the boundedness of $T_\Om$ from $L^2\times L^2$ to $L^1$.
\epf

\subsection{Rough bilinear singular integrals of R. Fefferman type} 
In the previous subsection we studied the rough bilinear singular integral with kernel 
$$
K(x)=\f{\Om(x')}{|x|^{2n}},
$$
which is a smooth function in the radial direction. Fefferman~\cite{Fefferman1979} observed that, in the linear case, smoothness of the kernel in the radial direction is unnecessary and obtained boundedness of the singular integral operator with kernel of the form
\begin{equation}\label{e11241}
K(x)=\rho(|x|)\f{\Om(x')}{|x|^{2n}},
\end{equation}
where $\rho$ is any bounded function. Let us now consider the bilinear operator $T_K$ associated with this kernel. Motivated by an extension of the above mentioned result due to Duoandikoetxea and Rubio de Francia~\cite{Duoandikoetxea1986}, we slightly relax the boundedness assumption on $\rho$ and assume that it satisfies the less restrictive condition
\begin{equation}\label{e11242}
\int_0^R|\rho(r)|^2dr\le C_\rho R.
\end{equation}

Our result is the following theorem.

\begin{thm}\label{T:fefferman}
Suppose that $\Om$ lies in $ L^r(\mathbb S^{2n-1})$ with $r>4/3$ and has vanishing integral over $\mathbb S^{2n-1} $. Let $\rho$ be a function on the real line satisfying~\eqref{e11242}. 
Then the bilinear singular integral operator $T_K$ with kernel $K$ given by~\eqref{e11241} satisfies
 $$
\|T_K\|_{L^{p_1}\times L^{p_2}\to L^p}<\nf
$$
whenever $2\leq p_1, p_2 \leq \infty$, $1\leq p\leq 2$ and $\frac{1}{p}=\frac{1}{p_1}+\frac{1}{p_2}$.

Moreover, if  $\Om\in L^\nf$ and $\rho\in L^\nf$, then $T_K$ is bounded from $L^{p_1}\times L^{p_2}$ to $L^p$ for $1<p_1,p_2<\nf$, $1/2<p<\infty$, and $ 1/p= 1/{p_1}+ 1/{p_2}$.
\end{thm}

To prove Theorem~\ref{T:fefferman} we need the following variant of Corollary~\ref{11245}.

\begin{prop}\label{11251}
Let $1\leq q <4$ and let $(M_k)_{k\in \Z}$ be a sequence of multipliers belonging to $\mc C^{M_q'}(\R^{2n})$, where $M_q'=\max\big(2n,\lfloor \frac{2n}{4-q}\rfloor +1\big)$, and satisfying 
\begin{equation}\label{E:integrability_M_k}
\sup_{k\in \Z} \|M_k(2^{-k}\cdot)\|_{L^q} <\infty.
\end{equation}
Assume, moreover, that
\begin{equation}\label{E:M_k}
|M_k(\xi,\eta)|\le C\min(|2^k(\xi,\eta)|, |2^k(\xi,\eta)|^{-\de}),
\end{equation}
and 
\begin{equation}\label{E:derivative_M_k}
|\p^\al M_k(\xi,\eta)|\le C_\al 2^{k|\al|}\min(1, |2^k(\xi,\eta)|^{-\de})
\end{equation}
for all multiindices $\alpha$ with $|\alpha|\leq M_q'$ and some fixed $\delta>0$.
Let $T_{M_k}$ be the bilinear operator associated with multiplier $M_k$, and define
$T=\sum_{k\in\bbz} T_{M_k}$.
Then
$$
\|T\|_{L^{p_1}\times L^{p_2}\to L^p}<\nf
$$
whenever $2\le p_1,p_2\le\nf$, $1\leq p\leq 2$ and $ 1/p= 1/{p_1}+ 1/{p_2}$.
\end{prop}

Proposition~\ref{11251} coincides with Corollary~\ref{11245} if $M_k(\xi,\eta)=m(2^k(\xi,\eta))$. In the general case, the proof of Corollary~\ref{11245} translates   verbatim into the proof of Proposition~\ref{11251}; we leave the details to the interested reader. 

We shall also need the following lemma which follows from the proof of \cite[Corollary 4.1]{Duoandikoetxea1986}.
\begin{lm}\label{11243}
Let $k\in \Z$ and let $K$ be as in Theorem~\ref{T:fefferman}. If $K^k=K\psi(2^{-k}\cdot)$, where $\psi$ is a smooth function supported in the unit annulus of $\R^{2n}$, and $\de$ is a positive real number satisfying $2\de r'<1$, then we have
$$
|\wh{K^k}(\xi,\eta)|\le C\|\Om\|_{L^r}\min(|2^k(\xi,\eta)|, |2^k(\xi,\eta)|^{-\de})
$$
and 
$$
|\p^\al \wh{K^k}(\xi,\eta)|\le C_\al \|\Om\|_{L^r}2^{k|\al|}\min(1, |2^k(\xi,\eta)|^{-\de})
$$
for all multiindices $\alpha$.
\end{lm}

\bpf[Proof of Theorem~\ref{T:fefferman}]
Let $\psi$ be a smooth function supported in the unit annulus of $\R^{2n}$ such that $\sum_{k\in \Z} \psi(2^{-k} \cdot) =1$. For $k\in \Z$ we denote $\psi_k=\psi(2^{-k}\cdot)$ and $M_k=\wh{K\psi_k}$. Then $T_K=\sum_{k\in \Z} T_{M_k}$, and the first statement will thus follow if we verify the assumptions~\eqref{E:M_k}, \eqref{E:derivative_M_k} and \eqref{E:integrability_M_k} of Proposition~\ref{11251}.

The validity of conditions~\eqref{E:M_k} and~\eqref{E:derivative_M_k} follows from Lemma~\ref{11243}. Let us now show that condition~\eqref{E:integrability_M_k} is fulfilled with $q=r'<4$. 
  We will assume throughout the proof that $q'=r\leq 2$. This assumption can be  made without loss of generality thanks to the embedding $L^{r_1}(\mathbb S^{2n-1})\subseteq L^{r_2}(\mathbb S^{2n-1})$ when
$r_1\ge r_2$. For any fixed $k\in \Z$ we have
\begin{equation*}
\|M_k(2^{-k}\cdot)\|_{L^q}
=2^{\frac{2kn}{q}} \|M_k\|_{L^q}
=2^{\frac{2kn}{q}} \|\wh{K\psi_k}\|_{L^q}
\lesssim 2^{\frac{2kn}{q}} \|K\psi_k\|_{L^{q'}},
\end{equation*}
where the last estimate follows from the Hausdorff-Young inequality. Now,
\begin{align*}
\|K\psi_k\|_{L^{q'}}^{q'}
&=\int_{\R^{2n}} |\rho(|x|)\psi_k(x)|^{q'} \frac{|\Omega(x')|^{q'}}{|x|^{2nq'}}\,dx\\
&\lesssim \int_{2^{k-1}}^{2^{k+1}} |\rho(r)|^{q'} r^{2n(1-q')-1} \int_{\mathbb S^{2n-1}} |\Omega(\theta)|^{q'}\,d\theta\,dr\\
&\lesssim 2^{k(2n(1-q')-1)} \|\Omega\|_{L^{q'}}^{q'} \int_{2^{k-1}}^{2^{k+1}} |\rho(r)|^{q'}\,dr\\
&\lesssim 2^{k(2n(1-q')-\frac{q'}{2})} \|\Omega\|_{L^{q'}}^{q'} \left(\int_{2^{k-1}}^{2^{k+1}} |\rho(r)|^{2}\,dr\right)^{\frac{q'}{2}}\\
&\lesssim 2^{2kn(1-q')} \|\Omega\|_{L^{q'}}^{q'},
\end{align*}
by~\eqref{e11242}.
Altogether,
$$
\|M_k(2^{-k}\cdot)\|_{L^q} \leq C 2^{\frac{2kn}{q}} 2^{-\frac{2kn}{q}} \|\Omega\|_{L^{q'}}=C\|\Omega\|_{L^r},
$$
with $C$ independent of $k$, as desired.

We now turn to the second statement. We denote by $T_j$ the bilinear operator associated with the multiplier 
$$
m_j(\xi,\eta)=\sum_{k\in\bbz}M_k(\xi,\eta)\psi(2^{k-j}(\xi,\eta)).
$$ 
The $L^{p_1} \times L^{p_2} \rightarrow L^p$ boundedness of $\sum_{j\leq 0} T_j$ follows from the bilinear Coifman-Meyer theory by an argument analogous to the one in~\cite[Proposition 3]{Grafakos2015}. Let us thus assume that $j>0$ in what follows. Then the operator $T_j$ is bounded from
$L^2\times L^2$ to $L^1$ with bound $C2^{-j\de}$, where $\de\sim 1/r'$.
If both $\Om$ and $\rho$ are bounded, following the argument in \cite[Section 6]{Grafakos2015} we deduce that $T_j$ is a bilinear Calder\'on-Zygmund operator with constant
$C_\varepsilon\, j^\varepsilon$ for any $\varepsilon \in (0,1)$. Interpolating between the two estimates as in~\cite[Lemma 12]{Grafakos2015}, we obtain
$\|T_j\|_{L^{p_1}\times L^{p_2}\to L^p}\le C2^{-j\de_1}$ with $\de_1=\de_1(p_1,p_2)>0$,
where $1<p_1,p_2<\nf$ and $\tf1p=\tf1{p_1}+\tf1{p_2}$.
Summing over $j$, the boundedness of $T$ follows.
\epf

\subsection{Bilinear dyadic spherical maximal operators}

In this subsection we study the bilinear dyadic  spherical maximal operator given by
\begin{equation}
A^d(f,g)(x)=\sup_{k\in \mathbb Z} |A_{2^k}(f,g)(x)|,
\end{equation}
where $A_t(f,g)(x)=\int_{\mathbb S^{2n-1}}f(x-ty)g(x-tz)d\si(y,z)$.
The $L^2 \times L^2 \rightarrow L^1$ boundedness of this operator follows from Theorem~\ref{04161}
by a routine argument.

\begin{thm}
The bilinear dyadic spherical maximal operator is bounded from
$L^{p_1}(\rn)\times L^{p_2}(\rn)$ to $L^p(\rn)$ whenever $n\geq 2$ and $2\leq p_1,p_2\leq \infty$, $1\leq p\leq 2$ and $ {1}/{p}= {1}/{p_1}+ {1}/{p_2}$. 
\end{thm}

\bpf
Let 
$$
\varphi(y,z)=\psi(y)\psi(z),
$$ 
where $\psi$ is a radially decreasing Schwartz function on $\R^n$ such that
$$
\left(\int_{\R^n} \psi(x)\,dx\right)^2=\int_{\R^{2n}} \varphi(y,z)\,dy\,dz= |\mathbb S^{2n-1}|.
$$ 
We define $\mu=d\si-\vp$ and observe that, by~\cite[Theorem 2.1.10]{CFA},
$$
A^d(f,g)(x)\le |\mathbb S^{2n-1}| M(f)(x)M(g)(x)+M_\mu(f,g)(x),
$$
where $M$ is the Hardy-Littlewood maximal operator and $M_\mu(f,g)(x)=\sup_{k\in\mathbb Z}|A_{\mu,k}(f,g)|(x)$ with
\begin{align*}
A_{\mu, k}(f,g)(x)=&\int_{\bbr^{2n}}f(x-2^ky)g(x-2^kz)d\mu(y,z)\\
=&\int_{\bbr^{2n}} \wh f(\xi)\wh g(\eta)\wh \mu(2^k(\xi,\eta))e^{2\pi ix\cdot(\xi+\eta)}d\xi d\eta.
\end{align*}
Let $m=\wh\mu$. Since both $\wh{d\sigma}$ and $\wh{\varphi}$ are continuously differentiable and  $\wh{d\sigma}(0)=\wh{\varphi}(0)$, Taylor's remainder theorem yields the estimate
$|m(\xi,\eta)| \lesssim |(\xi,\eta)|$ in a neighborhood of the origin. Furthermore, it is well-known that $m$ satisfies
$\eqref{e09181}$ with $\de=
\tf{2n-1}2$ (see \cite[p. 178]{Duoandikoetxea2000}), which implies, in particular, that $m\in L^q$ for every $q>\frac{4n}{2n-1}$. Thus, $m\in L^q$ for some $q<4$ provided that $n\geq 2$. 

Using Khintchine's inequality and Fubini's theorem, we can control $\|M_\mu(f,g)\|_{L^p}^p$ by
$$ 
\Big\|\Big(\sum_{k\in \Z} |A_{\mu,k}(f,g)|^2\Big)^{1/2}\Big\|_{L^p}^p
\approx \int_0^1\Big\|\sum_{k\in \Z} r_k(t) A_{\mu,k}(f,g)\Big\|_{L^p}^pdt,$$
where $\{r_k\}$ is the sequence of Rademacher functions.
Consequently, by Corollary~\ref{11245}, we obtain
that 
$$
\Big\|\sum_{k\in \Z} r_k(t) A_{\mu,k}(f,g)\Big\|_{L^p}\le 
C\|f\|_{L^{p_1}}\|g\|_{L^{p_2}}
$$
with $C$ independent of $t$.
This concludes the proof.
\epf




\end{document}